\documentclass[a4paper, 10pt]{amsart}
\usepackage[top=80pt,bottom=65pt,left=70pt,right=70pt]{geometry}
\usepackage{graphicx, eepic}
\usepackage{times}
\normalfont
\usepackage{bbm}
\usepackage{xcolor}
\usepackage{verbatim}
\usepackage{kotex}
\usepackage{mathrsfs}
\usepackage{amscd}
\usepackage{float}
\usepackage[all]{xy}
\usepackage[subnum]{cases}
\ifpdf
\usepackage{tikz}
\usepackage{textcomp}
\fi
\usetikzlibrary{arrows,matrix,positioning}

\usepackage{amsthm,amsmath,amsfonts,amssymb}
\usepackage[colorlinks=true, linkcolor=blue, citecolor=red, filecolor=violet, urlcolor=violet, backref=page]{hyperref}

\usepackage{multirow, url}
\usepackage{color}
\usepackage[all]{xy}
\theoremstyle{plain}
\newtheorem{theorem}{Theorem}[section]

\theoremstyle{definition}

\newtheorem{conjecture}[theorem]{Conjecture}

\theoremstyle{remark}

\newcommand{\C}{\mathbb{C}}

\newcommand{\R}{\mathbb{R}}

\newcommand{\Z}{\mathbb{Z}}

\newcommand{\vs}{\vspace}
\newcommand{\ds}{\displaystyle}

\numberwithin{equation}{section} \numberwithin{table}{section}

\begin{document}                                                                          

\title{Remark on the Betti numbers for Hamiltonian circle actions}
\author{Yunhyung Cho}
\address{Department of Mathematics Education, Sungkyunkwan University, Seoul, Republic of Korea. }
\email{yunhyung@skku.edu}



\begin{abstract}
	In this paper, we establish a certain inequality in terms of Betti numbers of a closed Hamiltonian $S^1$-manifold with isolated fixed points. 
\end{abstract}
\maketitle
\setcounter{tocdepth}{1} 

\section{Introduction}
\label{secIntroduction}

	Let $(M,\omega)$ be a $2n$-dimensional closed symplectic manifold admitting a Hamiltonian torus action with only isolated fixed points. It has been a long-standing open problem whether $M$ admits a 
	K\"{a}hler metric or not. Historically, Delzant \cite{De} proved that if $M$ admits a Hamiltonian $T^n$-action, where the fixed point set is automatically discrete, then $M$ admits a $T^n$-invariant 
	K\"{a}hler metric. 
	Restricting to an $S^1$-action case, several results on the existence of a K\"{a}hler metric were provided in some special cases. 
	For instance, Karshon \cite{Ka} proved that every closed symplectic four manifold admitting a Hamiltonian circle action 
	admits a K\"{a}hler metric. (In fact, the $S^1$-action is induced from a toric action when the fixed points are isolated.) Also if $\dim M = 6$ with $b_2(M)=1$, then 
	it turned out that $M$ admits a 
	K\"{a}hler metric, which was proved by Tolman \cite{T1} and McDuff \cite{McD}. 
	Recently, the author have shown that any 6-dimensional monotone closed semifree Hamiltonian $S^1$-manifold
	admits a K\"{a}hler metric, see \cite{Cho2, Cho3, Cho4}. 
	
	As a counterpart, there were ``candidates'' of closed Hamiltonian $T$-manifolds (with isolated fixed points) which possibly fail to admit K\"{a}hler metrics.
	Tolman \cite{T2} and Woodward \cite{W} constructed a six-dimensional closed Hamiltonian $T^2$-manifold with only isolated fixed points and with 
	no $T^2$-invariant K\"{a}hler metric.
	Surprisingly Goertsches-Kostantis-Zoller \cite{GKZ} have recently shown that examples of Tolman and Woodward 
	indeed admit K\"{a}hler metrics that are not $T^2$-invariant. Thus their result provides a positive evidence for the conjecture of the existence of K\"{a}hler metrics.

	On the other hand, it seems reasonable to ask whether $(M,\omega)$ enjoys K\"{a}hlerian properties, such as the hard Lefschetz property of the symplectic form $\omega$ or 
	the unimodality of even Betti numbers. 
	Recall that every closed K\"{a}hler manifold $(M,\omega, J)$ satisfies the {\em hard Lefschetz property}, that is, 
	\[
		\begin{array}{ccccl}
				[\omega]^{n-k} & : & H^k(M; \R) & \rightarrow & H^{2n-k}(M;\R) \\
							&	& \alpha & \mapsto & \alpha \cup [\omega]^{n-k}
		\end{array}
	\]
	is an isomorphism for every $k=0,1,\cdots, n$. This implies that 
	\[
		[\omega] : H^k(M;\R) \rightarrow H^{k+2}(M;\R)
	\] is injective for every $k$ with $0 \leq k < n$, and therefore the sequence of even (as well as odd) Betti numbers of $M$ is unimodal. In other words,
	\[
		b_k \leq b_{k+2}, \quad k = 0,1,\cdots, n-1
	\]
	where $b_i$ denotes the $i$-th Betti number of $M$. In this paper we deal with the following conjecture. 
	
	\begin{conjecture}\cite{JHKLM}\label{conj_unimodality}
		Let $(M,\omega)$ be a $2n$-dimensional closed symplectic manifold equipped with a Hamiltonian $S^1$-action with only isolated fixed points. Then the sequence of even Betti numbers is 
		{\em unimodal}, i.e., 
		\[
			b_{2i} \leq b_{2i + 2} \quad \text{for every $0 \leq i < \left \lfloor{\frac{n}{2}} \right\rfloor$.}
		\]
	\end{conjecture}	
	
	It is worth mentioning that every odd Betti number of $M$ vanishes by Frankel's theorem which states that a moment map is a Morse function whose critical points are of even indices. 
	(See \cite[Theorem IV.2.3]{Au}.) Therefore we only need to care about even Betti numbers of $M$.
	
	In \cite{CK1}, the author and Kim proved Conjecture \ref{conj_unimodality} when $\dim M = 8$. The main goal of this article is to improve the result of \cite{CK1} and prove the following
	inequality, which is automatically satisfied when Conjecture \ref{conj_unimodality} is true.
	
	\begin{theorem}\label{thm_main}
		Let $(M,\omega)$ be a closed symplectic manifold admitting a Hamiltonian circle action with only isolated fixed points where 
		$\dim M = 8n$ or $8n + 4$. . Then 
		\[
			b_2 + \cdots + b_{2 + 4(n-1)} \leq b_4 + \cdots + b_{4 + 4(n-1)}. 
		\]
		In particular when $\dim M = 8$ or $12$, we have 
		\[
			b_2 \leq b_4.
		\]
	\end{theorem}	
		
\subsection*{Acknowledgements} 
This work is supported by the National Research Foundation of Korea(NRF) grant funded by the Korea government(MSIP; Ministry of Science, ICT \& Future Planning) (NRF-2017R1C1B5018168).

\section{Proof of the main theorem}
\label{secProofOfTheMainTheorem}	
	
	The main technique for proving Theorem \ref{thm_main} is the ABBV-localization due to Atiyah-Bott and Berline-Vergne. 
	Recall that for an $S^1$-manifold $M$, the {\em equivariant cohomology} is defined by  $H^*_{S^1}(M) := H^*(M \times_{S^1} ES^1)$ where 
	$ES^1$ is a contractible space on which $S^1$ acts freely. Then $H^*_{S^1}(M)$ inherits an $H^*(BS^1)$-module structure induced from the projection 
	\[
		\pi : M \times_{S^1} ES^1 \rightarrow BS^1 := ES^1 / S^1. 
	\]
	Note that $H^*(BS^1;\R) \cong H^*(\C P^\infty;\R) = \R[u]$. Moreover, for the inclusion map $i : M^{S^1} \hookrightarrow M$, we have an induced ring homomorphism 
	\[
		i^* : H^*_{S^1}(M;\R) \rightarrow H^*_{S^1}(M^{S^1};\R) \cong H^*(BS^1;\R) \otimes H^*(M^{S^1};\R). 
	\]
	When $M^{S^1} = \{p_1, \cdots, p_m\}$ is discrete, we may express as $H^*(BS^1;\R) \otimes H^*(M^{S^1};\R) \cong \bigoplus_{i=1}^m H^*(BS^1;\R)$ and so 
	\[
		i^*(\alpha) = (f_1, \cdots, f_m), \quad f_i \in \R[u]
	\] 
	for $\alpha \in H^*_{S^1}(M;\R)$.
	We denote by $\alpha|_{p_i} := f_i$ and call it the {\em restriction of $\alpha$ to $p_i$}. By the Kirwan's injectivity theorem \cite{Ki}, the map $i^*$ is injective and hence
	$H^*_{S^1}(M;\R)$ is a free $H^*(BS^1;\R)$-module.	

	\begin{theorem}[ABBV Localization theorem]\cite{AB, BV}\label{thm_localization}
		Let $M$ be a closed $S^1$-manifold with only isolated fixed points and $\alpha \in H^*_{S^1}(M;\R)$. Then we have 
		\[
			\int_M \alpha = \sum_{p \in M^{S^1}} \frac{\alpha|_p}{\left(\Pi_{i=1}^n w_i(p) \right)u^n}.
		\]
		where $w_1(p), \cdots, w_n(p)$ denote the weights of the tangential $S^1$-representation at $p$.
	\end{theorem}
	
	To obtain Theorem \ref{thm_main}, we will apply Theorem \ref{thm_localization} to {\em canonical classes} which form a basis of $H^*_{S^1}(M;\R)$ as an $H^*(BS^1;\R)$-module.
	
	\begin{theorem}\cite{MT}\label{thm_canonical}
		Let $(M,\omega)$ be a $2n$-dimensional closed Hamiltonian $S^1$-manifold with only isolated fixed points. For each fixed point $p \in M^{S^1}$ of index $2k$, there exists 
		a unique class $\alpha_p \in H^{2k}_{S^1}(M; \Z)$ such that 
		\begin{itemize}
			\item $\alpha_p|_q = 0$ for every $q (\neq p) \in M^{S^1}$ with either $H(q) \leq H(p)$ or $\mathrm{ind}(q) \leq 2k$, 
			\item $\alpha_p|_p = \prod_{i=1}^k \lambda_i u$, where $\lambda_1, \cdots, \lambda_k$ are negative weights of the $S^1$-action at $p$.
		\end{itemize}	
		Moreover, the set $\{ \alpha_p ~|~ p \in M^{S^1} \}$ is a basis of $H^*_{S^1}(M;\R)$ as an $H^*(BS^1; \R)$-module.
	\end{theorem}
	
	Now we are ready to prove Theorem \ref{thm_main}. 
	
	\begin{proof}[Proof of Theorem \ref{thm_main}] 
		We first consider the case $\dim M = 8n$. Suppose that
		\begin{equation}\label{equ_assumption}
			b_2 + \cdots + b_{2 + 4(n-1)} > b_4 + \cdots + b_{4 + 4(n-1)}. 
		\end{equation}
		Since $H^*_{S^1}(M)$ is a free module over $H^*(BS^1)$, we have 
		\[
			H^{4n-2}_{S^1}(M) \cong u^0 \otimes H^{4n-2}(M) \oplus u^1 \otimes H^{4n-4}(M) \oplus \cdots \oplus u^{(2n-1)} \otimes H^{0}(M) 
		\]
		which implies that 
		\begin{itemize}
			\item $\dim_{\R} H^{4n-2}_{S^1}(M;\R) \cong b_0 + b_2 + \cdots b_{4n-2}$, and 
			\item $\{\alpha_p \cdot u^{2n - 1 - \frac{1}{2} \mathrm{ind}(p)}~|~ p \in M^{S^1}, ~\mathrm{ind}(p) \leq 4n-2 \}$ is a basis of $H^{4n-2}_{S^1}(M;\R)$ (as an $\R$-vector space) 
			by Theorem \ref{thm_canonical}.
		\end{itemize}

		Now, consider the following map
		\[
			\begin{array}{ccccl}\vs{0.2cm}
				\Phi & : & H^{4n-2}_{S^1}(M;\R) & \rightarrow &  \ds \left(\R^{b_0} \oplus \R^{b_4} \oplus 
				\cdots \oplus \R^{b_{4(n-1)}}\right) \oplus \left(\R^{b_{4n}} \oplus \cdots \oplus \R^{b_{8n-4}} \right) \\
						& 	& \alpha & \mapsto & \left(\alpha_0, \cdots, \alpha_{4n-4}, \alpha_{4n}, \cdots, \alpha_{8n -4} \right)
			\end{array}
		\]
		with the identification 
		\begin{equation}\label{equ_identification}
			\R^{b_{4i}} = \ds \bigoplus_{\mathrm{ind}(p) = 4i} \R \cdot u^{2n-1} \quad \text{and} \quad 
			\alpha_{4i} := \left( \alpha|_p \right)_{\mathrm{ind}(p) = 4i} \in \ds \bigoplus_{\mathrm{ind}(p) = 4i} \R \cdot u^{2n-1}
		\end{equation}
		for each $i=1,\cdots,n$. Since the dimension of the range of the map $\Phi$ satisfies
		\[
			\dim \mathrm{Im} \Phi \leq b_0 + \cdots + b_{4n-4} + (b_{4n} + b_{4n+4} + \cdots + b_{8n-4}) < b_0 + \cdots + b_{4n-4} + (b_{4n-2} + \cdots + b_2) 
		\]
		by the Poincar\'{e} duality and our assumption \eqref{equ_assumption}, the map $\Phi$ has a non-trivial kernel.
		In other words, there exists an element $\alpha \in H^{4n-2}_{S^1}(M;\R)$ such that 
		\[
			\alpha|_p = 0 
		\]
		for every fixed point $p \in M^{S^1}$ of index $0,4,\cdots, 8n-4$. 
		
		Now fix a moment map $H$ for the $S^1$-action on $(M,\omega)$ such that $H$ attains the maximum value $0$.  
		Denote by $p_{\max}$ the maximal fixed point and so $\mathrm{ind}(p_{\max}) = 8n$. 
		The equivariant extension $[\omega_H] \in H^2_{S^1}(M;\R)$ of $\omega$ with respect to the moment map $H$ satisfies 
		\[
			[\omega_H]|_{p} = -H(p)u \in \R[u]
		\]
		for every $p \in M^{S^1}$, see \cite[Proposition 2.6]{Cho1}. Since $H(p) < 0$ for every $p \neq p_{\max}$ by the choice of $H$, we obtain
		\begin{equation}\label{equ_localization}
			0 = \int_M \alpha^2 \cdot [\omega_H] = \sum_{p \in M^{S^1}} \frac{- \alpha^2|_p \cdot H(p) u}{(\prod_{i=1}^n w_i(p)) u^n} = \sum_{\mathrm{ind}(p) \equiv 2 ~(\mathrm{mod} 4)}
			\frac{- \alpha^2|_p \cdot H(p) u}{(\prod_{i=1}^n w_i(p)) u^n}
		\end{equation}
		by the ABBV localization theorem \ref{thm_localization} and the fact $[\omega_H]|_{p_{\max}} = -H(p_{\max})u = 0$. Moreover, there exists at least one fixed point $p \in M^{S^1}$ such 
		that 
		\[
			\alpha|_p \neq 0 \quad \text{and} \quad \mathrm{ind}(p) < 8n
		\]
		because
		\begin{itemize}
			\item $\alpha|_p \neq 0$ for some $p \in M^{S^1}$ by the Kirwan's injectivity theorem \cite{Ki}, and 
			\item if $\alpha|_p = 0$ for every $p \in M^{S^1}$ with $p \neq p_{\max}$, then $\alpha|_{p_{\max}} \neq 0$ and it violates the localization theorem \ref{thm_localization}
			\[
				0 = \int_M \alpha = \frac{\alpha|_{p_{\max}}}{(\prod_{i=1}^n w_i(p)) u^n} \neq 0.
			\]
		\end{itemize}
		Consequently, each summand of the rightmost equation of \eqref{equ_localization} has non-negative coefficient (of $\frac{1}{u}$) and at least one of those should be negative.
		Therefore it leads to a contradiction. 
		
		Now it remains to consider the case of $\dim M = 8n+4$. Under the same assumption \eqref{equ_assumption}, we similarly define 
		\[
			\begin{array}{ccccl}\vs{0.2cm}
				\Phi & : & H^{4n}_{S^1}(M;\R) & \rightarrow &  \ds \left(\R^{b_0} \oplus \R^{b_4} \oplus 
				\cdots \oplus \R^{b_{4n}}\right) \oplus \left(\R^{b_{4n+4}} \oplus \cdots \oplus \R^{b_{8n}} \right) \\
						& 	& \alpha & \mapsto & \left(\alpha_0, \cdots, \alpha_{4n}, \alpha_{4n+4}, \cdots, \alpha_{8n} \right)
			\end{array}
		\]
		with the same identification as in \eqref{equ_identification}. Note that 
		$\dim_{\R}  H^{4n}_{S^1}(M;\R) = b_0 + b_2 + \cdots + b_{4n-2} + b_{4n}$ and 
		\[
			\begin{array}{ccl}
				\dim \mathrm{Im} \Phi \leq b_0 + \cdots + b_{4n} + (b_{4n+4} + \cdots + b_{8n}) & = & b_0 + \cdots + b_{4n} + (b_{4n} + \cdots + b_{4}) \\
																										& < & b_0 + \cdots + b_{4n} + (b_{4n-2} + \cdots + b_{2}) = \dim_{\R}  H^{4n}_{S^1}(M;\R)\\
			\end{array}
		\]
		by the assumption \eqref{equ_assumption} and the Poincar\'{e} duality again. Thus $\Phi$ has a non-trivial kernel $\alpha \in H^{4n}_{S^1}(M;\R)$. In a similar manner as the previous case, we obtain 
		\[
			0 = \int_M \alpha^2 \cdot [\omega_H] = \sum_{\mathrm{ind}(p) \equiv 2 ~(\mathrm{mod} 4)} \frac{- \alpha^2|_p \cdot H(p) u}{(\prod_{i=1}^{n+1} w_i(p)) u^{n+1}} \neq 0
		\]
		which leads to a contradiction. This completes the proof.
	\end{proof}
	
\bibliographystyle{annotation}

\begin{thebibliography}{99}

	\bibitem[AB]{AB} M. Atiyah and R. Bott, {\em The moment map and equivariant cohomology}, Topology, \textbf{23} (1984), 1--28.
	
	\bibitem[Aud]{Au} Mich\'{e}le  Audin, {\em Topology of Torus actions on symplectic manifolds Second revised edition.} 
			Progress in Mathematics \textbf{93}, Birkh\"{a}user Verlag, Basel (2004).

	\bibitem[BV]{BV} N. Berline and M. Vergne, {\em Classes charact\'{e}ristiques \'{e}quivariantes. Formule de localisation en cohomologie \'{e}quivariante}, C. R. Acad. Sci. Paris S\'{e}r. I. Math. \textbf{295} 
	(1982), 539--541.
				
	\bibitem[CK1]{CK1} Y. Cho and M-K. Kim, {\em Unimodality of the Betti numbers for Hamiltonian circle action with isolated fixed points}, Math. Res. Lett., \textbf{21} (2014), no. 4, 691--696.			
	
	\bibitem[CK2]{CK2} Y. Cho and M-K. Kim, {\em Hard Lefschetz property for Hamiltonian circle action with self-indexing moment maps}, Math. Res. Lett., \textbf{23} (2016), no. 3, 719--732.

	\bibitem[CK3]{CK3} Y. Cho and M-K. Kim, {\em Hard Lefschetz property for Hamiltonian torus actions on 6-dimensional GKM-manifolds}, J. Symplectic Geom., \textbf{16} (2018), no. 6, 1549--1590.

	\bibitem[Cho1]{Cho1} Y. Cho, {\em Unimodality of Betti numbers for Hamiltonian circle actions with index-increasing moment maps}, Int. J. Math., \textbf{27} (2016), no. 5, 1650043.
	
	\bibitem[Cho2]{Cho2} Y. Cho, {\em Classification of six dimensional monotone symplectic manifolds admitting semifree circle actions I,} Int. J. Math., \textbf{30} (2019), no. 6, 1950032.		
	
	\bibitem[Cho3]{Cho3} Y. Cho, {\em Classification of six dimensional monotone symplectic manifolds admitting semifree circle actions II,} arXiv:1904.10962.
	
	\bibitem[Cho4]{Cho4} Y. Cho, {\em Classification of six dimensional monotone symplectic manifolds admitting semifree circle actions III,} arXiv: 1905.07292.

	\bibitem[De]{De} T. Delzant, {\em Hamiltoniens p\'{e}riodiques et images convexes de l’application moment}, Bull. Soc. Math. France \textbf{116} (1988), no. 3, 315--339.

	
	\bibitem[GKZ]{GKZ} Oliver Goertsches, Panagiotis Konstantis, and Leopold Zoller, {\em GKM theory and Hamiltonian non-K\"{a}hler actions in dimension 6}, arXiv:1903.11684v1.


	\bibitem[JHKLM]{JHKLM} L. Jeffrey, T. Holm, Y. Karshon, E. Lerman, and E. Meinrenken, {\em Moment maps in various geometries}, available at 
	http://www.birs.ca/workshops/2005/05w5072/report05w5072.pdf

	\bibitem[Ka]{Ka} Y. Karshon, {\em Periodic Hamiltonian flows on four-dimensional manifolds.} Mem. Amer. Math. Soc. \textbf{141} (1999), no. 672.
			
	\bibitem[Ki]{Ki} F.C. Kirwan, {\em Cohomology of quotients in symplectic and algebraic geometry,} Mathematical Notes, \textbf{31}, Princeton University Press, Princeton,
	NJ, 1984.			
			
	\bibitem[McD]{McD} D. McDuff, {\em Some 6-dimensional Hamiltonian $S^1$-manifolds}, J. Topol. \textbf{2} (2009), 589--623.
	
	\bibitem[MT]{MT} D. McDuff and S. Tolman, {\em Topological properties of Hamiltonian circle actions,} Int. Math. Res. Pap. 2006(2006), 1--77.
	
	\bibitem[T1]{T1} S. Tolman, {\em On a symplectic generalization of Petrie's conjecture}, Trans. Amer. Math. Soc. \textbf{362} (2010), no. 8, 3963--3996.

	\bibitem[T2]{T2} S. Tolman, {\em Examples of non-K\"{a}hler Hamiltonian torus actions}, Invent. Math. \textbf{131} (1998), 299--310.
	
	\bibitem[W]{W} C. Woodward, {\em Multiplicity-free Hamiltonian actions need not be K\"{a}hler}, Invent. Math., \textbf{131} (1998), no. 2, 311--319.
	
\end{thebibliography}

\end{document}